\documentclass[12pt]{article}

\usepackage{fullpage}
\usepackage{amsmath}
\usepackage{amssymb}
\usepackage{amsthm} 

\usepackage{enumerate}

\newcommand{\rn}{\ensuremath{\mathbb{R}^n}}
\newcommand{\rbar}{\ensuremath{\overline{\mathbb{R}}}}
\newcommand{\mabomega}{\ensuremath{M(\alpha, \beta, \Omega)}}
\newcommand{\mcdomega}{\ensuremath{M(c, d, \Omega)}}
\newcommand{\ltoo}{\ensuremath{L^2(\Omega)}}
\newcommand{\lr}{\ensuremath{L^r(\Omega)}}
\newcommand{\hoz}{\ensuremath{H^1_0(\Omega)}}
\newcommand{\borelmeas}{\ensuremath{\mathcal{B}(\Omega)}}
\newcommand{\linfty}{\ensuremath{L^\infty(\Omega)}}
\newcommand{\wlq}{\ensuremath{W_0^{1,q}(\Omega)}}
\newcommand{\hmo}{\ensuremath{H^{-1}(\Omega)}}

\newcommand{\ve}{\ensuremath{\varepsilon}}
\newcommand{\ueps}{\ensuremath{u_\varepsilon}}
\newcommand{\jeps}{\ensuremath{J_\varepsilon}}
\newcommand{\optcon}{\ensuremath{\theta_\varepsilon^*}}
\newcommand{\optstat}{\ensuremath{u_\varepsilon^*}}
\newcommand{\mueps}{\ensuremath{\mu_\varepsilon}}

\newcommand{\optmu}{\ensuremath{\mu_\varepsilon^*}}
\newcommand{\theps}{\ensuremath{\theta_\varepsilon}}
\newcommand{\allq}{\ensuremath{q\in \left[1,1^\star\right)}}
\newcommand{\allr}{\ensuremath{r\in \left[1,\frac{n}{n-2}\right)}}
\newcommand{\aeps}{\ensuremath{A_\varepsilon}}
\newcommand{\lambdaeps}{\ensuremath{\lambda_\varepsilon}}
\newcommand{\veps}{\ensuremath{v_\varepsilon}}
\newcommand{\beps}{\ensuremath{B\left(x,\frac{x}{\varepsilon}\right)}}
\newcommand{\peps}{\ensuremath{P\left(x,\frac{x}{\varepsilon}\right)}}
\newcommand{\Phieps}{\ensuremath{\Phi\left(x,\frac{x}{\varepsilon}\right)}}
\newcommand{\taueps}{\ensuremath{\tau\left(x,\frac{x}{\varepsilon}\right)}}
\newcommand{\aepsper}{\ensuremath{A\left(x,\frac{x}{\varepsilon}\right)}}
\newcommand{\bdry}{\ensuremath{\partial\Omega}}

\numberwithin{equation}{section}

\newtheorem{lemma}{Lemma}[section]
\newtheorem{prop}[lemma]{Proposition}
\newtheorem{thm}[lemma]{Theorem}
\newtheorem{cor}[lemma]{Corollary}
\newtheorem{conj}{Conjecture}
\newtheorem{defn}[lemma]{Definition}

\theoremstyle{definition}
\newtheorem{rmrk}[lemma]{Remark}

\title{Homogenization of Some Low-Cost Control Problems}
\author{Rajesh Mahadevan and T. Muthukumar\\ Facultad de Ciencias
F\'{i}sicas Y Matem\'{a}ticas,\\
Universidad de Concepci\'{o}n,\\ Concepci\'{o}n, Chile\\ e-mail: {\tt
[mthirumalai,rmahadevan]@udec.cl}}

\begin{document}
\maketitle

\begin{abstract}
The aim of this article is to study the asymptotic behaviour of
some low-cost control problems. These problems motivate the study
of $H$-convergence with weakly converging data. An improved lower
bound for the limit of energy functionals corresponding to weak
data is established, in the periodic case. This fact is used to
prove the $\Gamma$-convergence of a low-cost problem with
Dirichlet-type integral. Finally, we study the asymptotic
behaviour of a low-cost problem with controls converging to
measures.
\end{abstract}


\vspace{3mm}

\section{Introduction}

Let $n\ge 1$ and let $\Omega$ be a bounded open subset of $\rn$. Let
$0<\alpha\le\beta$ be two given positive real constants. We denote
by $\mabomega$ the class of all $n\times n$ matrices $A$ with
entries in $L^\infty(\Omega)$, such that,
\[
\alpha|\xi|^2 \leq A(x)\xi\cdot\xi \leq \beta|\xi|^2 \quad \text{ a.e.
in } x, \quad \forall \xi=(\xi_i)_{i=1}^n \in \rn.
\]
Let $A^t(x)$
denote the transpose of $A(x)$.

Given $A\in\mabomega$, 
$f\in\ltoo$, $N>0$ (a constant) and $U$ a closed convex subset of
$\ltoo$, we consider the following basic optimal control problem:
find $\theta^*\in U$ such that,
\begin{equation}\label{optctrl}
J(\theta^*)=\min_{\theta\in U}J(\theta),
\end{equation}
where the cost functional, $J: U \rightarrow \mathbb{R}$, is defined by
\begin{equation}\label{fcostfn}
J(\theta) = I(u,\theta) + \frac{N}{2}\|\theta\|_2^2 
\end{equation}
and the state $u=u(\theta)$ is the weak solution in $\hoz$ of the
boundary value problem 
\begin{equation}\label{fsteqn}
\left\{ \begin{array}{rll}
-\textup{div}(A\nabla u) &  =  f + \theta & \textup{in $\Omega$}\\
u &  =  0  & \textup{on $\partial\Omega$}.
\end{array}\right.
\end{equation}
We consider the following kinds of $I(u,\theta)$:
\begin{enumerate}[(a)]
\item For $B\in\mcdomega$ and symmetric, we consider Dirichlet-type
  integrals: 
\[
I(u,\theta)=\frac{1}{2}\int_\Omega B(x) \nabla
u(\theta)\cdot\nabla u(\theta)\,dx
\]

\item For a fixed $\allr$,
\[
I(u,\theta) = \|u(\theta)\|_r^r.
\]
\end{enumerate}
Then, it can be shown that the $J$ defined above are lower
semicontinuous, coercive and strictly convex, and, therefore, there
exists a unique optimal control, $\theta^*\in U$ minimizing $J$ over
$U$ (cf. \cite{dalmaso}). 

The main aim of this paper is to study the asymptotic behaviour of the
optimal control problems (\ref{optctrl}) depending on a small
parameter $\ve >0$ which represents the scale of heterogeneity of the
material. 

This problem was first studied in \cite{skmv,sksjphocp} for
varying coefficients $\{A_\ve\}\subset\mabomega$ in
\eqref{fsteqn}, $\{B_\ve\}\subset\mcdomega$ in the cost \eqref{fcostfn} corresponding to the case (a) and $N>0$ fixed. The approach used therein
consists in passing to the limit in the corresponding system of optimality
conditions involving an adjoint state. A complete characterization
of the asymptotic behaviour of the control problems was obtained.  

In the above problem, if $N$ is allowed to vary and degenerate by
taking $N=\ve$, then it is called {\em low-cost control} problem.
The low-cost control problems were first introduced by J.~L.~Lions
in \cite{rmrkchcon} (also see \cite{jllmma}) and extensively
studied in \cite{sksjplccp,tmk,sktmk2,tmkakn}. The corresponding
sequence of functionals $\jeps$ is not equicoercive over $\ltoo$ and thus the sequence
of optimal controls $\{\optcon\}$ is not bounded, {\em a priori}, in
$\ltoo$. But $\optcon$ is weakly compact in $\hmo$ and thus converges
to some $\theta^*\in\hmo$. This weak convergence is, in general,
not enough for studying the asymptotic behaviour of the system which
constitutes the optimality conditions. This case was studied in
\cite{sktmk2,tmkakn} and a partial homogenization of the
optimality system was proved when the control set is the
positive cone, with or without periodicity assumptions on the
coefficient matrices. In this article (cf. \S\ref{sec:bprob}), we
obtain, for the first time, a complete homogenization result for
low-cost control problems by taking $U$ to be an arbitrary closed
convex set in $\ltoo$ while assuming the  coefficients to be
periodic. Here we prove the variational convergence of the
optimal control problem in the framework of $\Gamma$-convergence. 

Subsequently, in \S\ref{sec:optmsr}, we study the
asymptotic behaviour of a low cost problem whose limit optimal control
will be in the space of measures. Given a constant $k>0$, let
\[
U =\left\{\theta\in\ltoo\mid \|\theta\|_1 \le k\right\}
\]
be the set of all admissible controls. We consider the optimal
control problem with cost functional as in case (b) governed by
\eqref{fsteqn} with varying coefficients $A_\ve\in\mabomega$. The main
difficulty with this problem, in contrast to the problem with
costs as in case (a), is that there is no weak compactness of the
optimal controls $\optcon$, even in $\hmo$. Thus, one is unable to
homogenize the control problem for a general admissible set
$U$. However, this problem was homogenized when $U$ is the
positive cone and $r=2$, in \cite{sktmk2, tmkakn}, and the limit problem was
obtained on the positive cone of $\hmo$.
We know, by Riesz representation theorem, that any non-negative
distribution in $\hmo$ is a non-negative Radon measure on $\Omega$. Thus, we wish to
consider the controls as measures. Therefore, in \S
\ref{sec:optmsr}, we consider the control set $U$ to be the class
of all functions in $\ltoo$ that are bounded in $L^1(\Omega)$ and
homogenize with respect to weak-* convergence of measures.

The paper is organised as follows: In \S\ref{sec:prem} we recall
some basic facts and tools required for the results proved in
\S\ref{sec:energy} and \S\ref{sec:bprob}. In \S\ref{sec:energy},
we conjecture on the best lower bound of `generalised' energy
functionals for weakly converging data and prove the same under
periodicity assumptions on the coefficients. In \S\ref{sec:bprob},
we homogenize the periodic low-cost control problems with cost as
in case (a). In \S\ref{sec:measdt}, we present the notion of
solution for measure data introduced by G.~Stampacchia. We also
give a $G$-convergence result with respect to varying measures.
Finally, in \S\ref{sec:optmsr}, we study the asymptotic behaviour
of the low-cost control problems with cost as in case (b).

\section{Preliminaries}\label{sec:prem}

In this section, we introduce some basic tools and facts that will
be used in this article. 

\subsection{$G$-convergence and $H$-convergence}\label{sec:hcon}

For all the results in this section we refer to \cite{murtar, murat1,
tartar}. We say a sequence $\{A_\ve\}\subset \mabomega$ {\bf $G$-converges}
to $A_0$ (denoted as $A_\ve \stackrel{H}{\rightharpoonup} A_0$)
iff for any $g\in\hmo$, the solution $\veps$ of 
\begin{equation}\label{eq:soe}
\left\{
\begin{array}{rll}
-\textrm{div}(A_\ve\nabla\veps) &  =  g & \textrm{in $\Omega$}\\
\veps &  =  0  & \textrm{on $\partial\Omega$}
\end{array}
\right.
\end{equation} is such that
\begin{equation}\label{eq:staweakcon}
\veps \rightharpoonup v_0 \text{ weakly in } \hoz 
\end{equation}
where $v_0$ is the unique solution of
\begin{equation}\label{eq:limsoe}
\left\{
\begin{array}{rll}
-\textup{div}(A_0\nabla v_0) &  =  g & \textrm{in $\Omega$}\\
v_0 &  =  0  & \textrm{on $\partial\Omega$}.
\end{array}
\right.
\end{equation}
The matrix $A_0$ is called the $G$-limit of the sequence
$\{A_\ve\}$. We say a sequence of matrices {\bf $H$-converges} to
$A_0$, if it $G$-converges and, in addition, for any $g\in\hmo$
we have 
\begin{equation}\label{eq:fluxcon}
A_\ve\nabla\veps \rightharpoonup A_0\nabla v_0 \text{
weakly in } (L^2(\Omega))^n
\end{equation}
where $\veps$ and $v_0$ are as in \eqref{eq:soe} and
\eqref{eq:limsoe}, respectively. For symmetric matrices both the notion coincide.

Given a sequence $\{A_\ve\}\subset\mabomega$ which $H$-converges
to $A_0$, the sequence of corrector matrices $P_\ve$ is that which
satisfies the following properties:
\begin{enumerate}[(a)]
\item $P_\varepsilon \rightharpoonup I
\textrm{ weakly in }
(\ltoo)^{n\times n}$.
\item $A_\ve P_\ve \rightharpoonup A_0 \textrm{ weakly in
} (\ltoo)^{n\times n}$.
\item ${^tP}_\ve A_\ve P_\ve \rightharpoonup
A_0$ weak* in $[\mathcal{D}^\prime(\Omega)]^{n\times n}$, the
space of distributions.
\end{enumerate}
One procedure to obtain the corrector matrix is by considering
$\chi^i_\ve\in H^1(\Omega)$, for $1\le i \le n$, which are
solutions of
\begin{equation}\label{eq:correcdef}
\left\{
\begin{array}{rll}
-\textrm{div}(A_\ve\nabla\chi^i_\ve) &  = -\textrm{div}(A_0 e_i) & \textrm{in $\Omega$}\\
\chi^i_\varepsilon &  =  x_i  & \textrm{on $\partial\Omega$}
\end{array}
\right.
\end{equation}
and then by defining $P_\ve e_i = \nabla \chi^i_\ve$ for $1 \le i \le n$.

\subsection{$\Gamma$-convergence}

For all the results in this section we refer to
\cite{dalmaso,braides}.
Let $X$ be a topological space and let $\rbar=\mathbb{R}
\cup \{-\infty,+\infty\}$. Let $\{F_m\}$ be a sequence of
functions from $X$ in to $\rbar$.

For any function $F$ on $X$, let $S(F)$ be the set of all lower
semicontinuous functions $G$ on $X$ such that $G\le F$. We define
the \emph{lower semicontinuous envelope} of $F$, $\overline{F}$, as
\[ \overline{F}(x)= \sup_{G\in S(F)}G(x), \quad \forall x\in X.\]
Observe that every lower semicontinuous
function is its own envelope.

We now define the \emph{sequential} $\Gamma$- lower and upper limit. The
$\Gamma$-\emph{upper limit} of $F_m$ is given by,
\[F^+(x):= \inf\left\{\limsup_{m\to \infty}F_m(x_m): x_m \to
x\right\}.
\]
Similarly, the $\Gamma$-\emph{lower limit} of $F_m$ is given by,
\[
F^-(x):= \inf\left\{\liminf_{m\to \infty}F_m(x_m): x_m \to x\right\}.
\]
We say a function $F$ is the $\Gamma$-limit of $F_m$ if
$F=F^+=F^-$. A characterization of the sequential $\Gamma$-limit $F$ w.r.t
the topology of $X$ is given by the following two conditions:
\begin{enumerate}[(i)]
\item For every $x\in X$ and for every sequence $\{x_m\}$
converging to $x$ in $X$, we have
\[ \liminf_{m\rightarrow\infty} F_m(x_m)\geq F(x). \]
\item For every $x\in X$, there exists a sequence $\{x_m\}$
converging to $x$ in $X$ such that
\[ \limsup_{m\rightarrow\infty} F_m(x_m)\le F(x). \]
\end{enumerate}

We now recall a result of $\Gamma$-convergence theory.
\begin{thm}\label{thm:gamma}
Let $F_m$ $\Gamma$-converge to $F$ in $X$ and let $x_m$ be a
minimizer of $F_m$ in $X$. If $\{x_m\}$ converges to $x$ in $X$,
then $x$ is a minimizer of $F$ in $X$ and the minima converge,
\[
F(x)=\lim_{m\to \infty}F_m(x_m).
\]\qed
\end{thm}

\subsection{Two-scale convergence}

For all the results in this section we refer to
\cite{guetseng,allaire2s,lukguet}.
Let $Y=(0,1)^n$ be the unit cell of $\rn$. We say that a sequence of
functions $\{\veps\}$ in $L^2(\Omega)$ \emph{weakly two-scale
converges} to a limit $v\in L^2(\Omega\times Y)$
(denoted as $v_\ve \stackrel{2\text{s}}{\rightharpoonup} v$) if
\[
\int_\Omega \veps(x) \phi\left(x,\frac{x}{\ve}\right)\,dx \rightarrow
\int_\Omega \int_Y v(x,y)\phi(x,y)\,dy\,dx, \quad \forall \phi\in
L^2[\Omega; C_{\text{per}}(Y)].
\]

It is possible to replace $L^2[\Omega; C_{\text{per}}(Y)]$ by
$\mathcal{D}[\Omega,C^\infty_{\text{per}}(Y)]$ in the above
definition of weak two-scale convergence provided we add the
assumption that $\{v_\ve\}$ is bounded in $L^2(\Omega)$.

We say that a sequence of functions $\{\veps\}$ in $L^2(\Omega)$
\emph{strongly} two-scale converges to $v\in L^2(\Omega\times Y)$
(denoted as $\veps\stackrel{2\text{s}}{\rightarrow} v$), if $\{v_\ve\}$ weakly
two-scale converges to $v$ and
\[
\lim_{\ve \rightarrow 0}\int_\Omega |v_\ve|^2 \,dx = \int_\Omega
\int_Y |v(x,y)|^2\,dy\,dx.
\]
We say that a function $\phi(x,y)$, $Y$-periodic in $y$, is an
\emph{admissible} function if the sequence
$\phi_\ve=\phi\left(x,\frac{x}{\varepsilon}\right)$ strongly
two-scale converges to $\phi(x,y)$.  The spaces
$C[\Omega;L^2_{\text{per}}(Y)]$,
$L^\infty[\Omega;C_{\text{per}}(Y)]$ and
$L^\infty_{\text{per}}[Y;C(\bar{\Omega})]$ are some examples of
classes of admissible functions.

We now state some of the main results of two-scale convergence
theory that will be used in this article.
\begin{thm}
For any bounded sequence $\{\veps\}\subset \ltoo$, there exists a
$v\in L^2(\Omega\times Y)$ such that, $\veps$ weakly two-scale
converges to $v$, for a subsequence. Also, if $\veps$ is bounded
in $H^1(\Omega)$, then $v$ is independent of $y$ and is in
$H^1(\Omega)$, and there exists a $v_1\in L^2[\Omega;
H^1_{\textup{per}}(Y)]$ such that, up to a subsequence, $\nabla
\veps$ weakly two-scale converges to $\nabla v + \nabla_y v_1$.
\end{thm}
\begin{thm}\label{thm:strweak}
Let $u_\ve \stackrel{2\text{s}}{\rightarrow} u$. Then, given any
sequence $v_\ve \stackrel{2\text{s}}{\rightharpoonup} v$, we have that
\[
\int_\Omega u_\ve(x) v_\ve(x) \tau(x)\,dx \rightarrow \int_\Omega
\int_Y u(x,y) v(x,y) \tau(x)\,dy\,dx
\]
for every $\tau\in C^\infty_0(\Omega)$.
The $\tau\in C^\infty_0(\Omega)$ can be
replaced with $\taueps$ where
$\tau\in\mathcal{D}(\Omega,C^\infty_{\textup{per}}(Y))$.
\end{thm}
We now recall a property of convex periodic functionals with
respect to two-scale convergence.
\begin{prop}[{\cite[Proposition 2.5]{bouchfragala}}]\label{prop:convexlsc}
Let $j:= j(y,\xi)$ be a measurable function on $\rn\times\rn$,
$Y$-periodic in $y$, convex in $\xi$ and satisfies for some constants $a,b>0$,
\[
a|\xi|^2 \le j(y,\xi) \le b(1+|\xi|^2),\quad (y,\xi)\in
\rn\times\rn.
\]
Also, let $\{v_\ve\}\subset\ltoo$ be such that
$\veps\stackrel{2\text{s}}{\rightharpoonup} v\in L^2(\Omega\times
Y)$. Then
\[
\liminf_{\ve\to 0}\int_\Omega j\left(\frac{x}{\ve},
\veps(x)\right)\,dx \ge \int_{\Omega\times Y} j(y,v(x,y))\,dx\,dy.
\]\qed
\end{prop}

\subsection{Convex Analysis}\label{sec:convex}

For any function $h$ on $\rn$, one defines its
convex conjugate $h^\prime$ on $\rn$ in the following way:
\[
h^\prime(x^\prime) = \sup_{x\in\rn}\left\{\langle
x,x^\prime\rangle - h(x)\right\}.
\]
The following inequality is called the \emph{Fenchel's inequality} for
any proper convex function $h$ and its conjugate $h^\prime$:
\[
\langle x,x^\prime\rangle \le h(x)+h^\prime(x^\prime), \quad
\forall x,x^\prime\in\rn.
\]
If $h$ is a quadratic convex function, say of the form
\[
h(x) = \frac{1}{2}\langle x,Qx\rangle
\]
where $Q$ is a symmetric, positive definite
$n\times n$ matrix, then
\begin{equation}\label{eq:conconj}
h^\prime(x^\prime) = \frac{1}{2}\langle
x^\prime,Q^{-1}x^\prime\rangle.
\end{equation}

For above results, we refer to \cite{rockafellar}. We now comment
on a classical property of commutativity of infimum and the
integral. The first results along this direction was proved by
Rockafellar in \cite{rockafellar2, rockafellar3}. Another version
of the same was proved in \cite[Theorem 1]{bouchvaladier}. However
we shall now state the version as given in \cite[Lemma
4.3]{bouchdalmaso}.

If $\{\Delta_k\}$ is a family of measurable set functions from
$\Omega$ into $\rn$, then there exists a measurable set function
(cf. \cite[Proposition 14]{valadier}) $\Delta$ from $\Omega$ into
$\rn$ with the following properties:
\begin{enumerate}[(i)]
\item For every $k$, we have $\Delta_k(x)\subseteq \Delta(x)$ for
a.e. $x\in\Omega$.
\item If $\Pi$ is a set function on $\Omega$ such that for every
$k$, $\Delta_k(x)\subseteq \Pi(x)$ for a.e. $x\in\Omega$, then
$\Delta(x)\subseteq \Pi(x)$ for a.e. $x\in \Omega$. 
\end{enumerate}
The set function $\Delta$ is denoted as $\Delta(x)=
\text{ess-sup}_{k}\Delta_k(x)$.

Let $E$ be a set of measurable functions from $\Omega$ to $\rn$.
We say $E$ is $C^1$-\emph{stable} if for every finite family
$\{\omega_k\}_k \subset E$ and for every non-negative family of
functions $\{\psi_k\}_k \subset C^1(\overline{\Omega})$ such that
$\Sigma_k \psi_k =1$ in $\Omega$, we have that $\Sigma_k
\psi_k\omega_k \in E$. Observe that $C^1$-stability implies
convexity.

\begin{lemma}[{\cite[Lemma 4.3]{bouchdalmaso}}]\label{lm:commute}
Let $E$ be a $C^1$-stable set and let $j$ be Borel measurable
on $\Omega\times\rn$ such that $j(x,\cdot)$ is convex on $\rn$
for a.e. $x\in\Omega$. Suppose that $j(\cdot,\omega(\cdot))\in L^1(\Omega)$, for
every $\omega\in E$, and let $\Delta(x)= \text{ess-sup}_{\omega\in
E}\{\omega(x)\}$ then
\[
\inf_{\omega\in E}\int_\Omega j(x,\omega(x))\,dx = \int_\Omega
\inf_{\zeta\in\Delta(x)} j(x,\zeta)\,dx.
\]\qed
\end{lemma}

\section{Energy bounds}\label{sec:energy}

Given $A_\ve \subset \mabomega$ which $H$-converges to $A_0$, a standard result of $H$-convegence is that the energies converge, \emph{i.e.},
\begin{equation}\label{eq:energy}
\int_\Omega A_\ve\nabla\veps.\nabla\veps\,dx
\stackrel{\ve\to 0}{\longrightarrow}
\int_\Omega A_0\nabla v_0.\nabla v_0\,dx,
\end{equation}
where $v_\ve$ and $v_0$ are the solution of \eqref{eq:soe} and
\eqref{eq:limsoe}, respectively. Moreover, one has from the theory of
$\Gamma$-convergence, the following basic result.
\begin{lemma}[{cf. \cite[Chapter 13]{dalmaso}}]\label{lm:murat} Given a sequence of symmetric
matrices $A_\ve \subset \mabomega$ which $G$-converges to $A_0$ and given any
sequence $w_\ve$ weakly converging to $w_0$ in $\hoz$, we have
\begin{equation}\label{eq:glltenergy}
\liminf_{\ve \rightarrow 0}\int_\Omega A_\ve \nabla
w_\ve\cdot\nabla w_\ve\, dx \ge \int_\Omega A_0\nabla
w_0\cdot\nabla w_0\,dx.
\end{equation}
\end{lemma}
Given $\{B_\ve\}\subset\mcdomega$, a sequences of symmetric
matrices and $v_\ve$ solutions of \eqref{eq:soe}, we have from Lemma~\ref{lm:murat} that,
\begin{equation}\label{eq:orineq}
\liminf_{\ve \rightarrow 0}\int_\Omega B_\ve \nabla
v_\ve\cdot\nabla v_\ve\, dx \ge \int_\Omega B_0\nabla v_0
\cdot\nabla v_0\,dx,
\end{equation}
where $v_0$ is solution of \eqref{eq:limsoe} and $B_0$ is the
$H$-limit of $\{B_\ve\}$. Moreover, \eqref{eq:orineq} remains
valid when $v_\ve$ are solutions of \eqref{eq:gvardata}, where
$g_\ve$ converges strongly to $g$ in $\hmo$. A question of
interest is to obtain the best lower bound for the limit on the
left hand side of \eqref{eq:orineq} when $v_\ve$ is the
solution of \eqref{eq:gvardata} and $A_\ve$ $H$-converges to $A_0$. In order to state a result
improving \eqref{eq:orineq} we recall (from \S\ref{sec:hcon})
that $\{P_\ve\}$ is the sequence of corrector matrices associated
with $\{A_\ve\}$. Let $B^\sharp$ be the weak-* limit of $\{P^t_\ve
B_\ve P_\ve\}$ in $(\linfty)^{n\times n}$.
\begin{thm}[cf. \cite{skmr}]\label{thm:genenercon}
Let $g_\varepsilon\rightarrow g$ strongly in $\hmo$ and let $v_\varepsilon\in\hoz$ be the weak solution of 
\begin{equation}\label{eq:gvardata}
\left\{ \begin{array}{rll}
-\textup{div}(A_\ve\nabla v_\ve) &  =  g_\ve & \textup{in
$\Omega$}\\
v_\ve &  =  0  & \textup{on $\bdry$},
\end{array}\right.
\end{equation}
then
\begin{equation}\label{eq:bgcon}
\int_\Omega B_\ve \nabla v_\ve\cdot\nabla v_\ve\, dx \stackrel{\ve\to 0}{\longrightarrow}
\int_\Omega B^\sharp\nabla v_0 \cdot\nabla v_0\,dx
\end{equation}
where $v_0\in\hoz$ is the unique solution of
\eqref{eq:limsoe}.\qed
\end{thm}
Comparing \eqref{eq:bgcon} with \eqref{eq:orineq} we observe that
$B^\sharp \ge B_0$. If $B_\ve=A_\ve$, then we have
\mbox{$B^\sharp=B_0=A_0$}. For
the properties of $B^\sharp$ we refer to \cite{skmr}.

\begin{rmrk}\label{rm:econ}
For $g_\ve$ weakly converging to $g$ in $\hmo$, although $v_\ve$
converges weakly in $\hoz$ (up to a subsequence) one can still not
conclude that \eqref{eq:orineq} holds. In fact, even the convergence
of the energies as in \eqref{eq:energy} is not valid.\qed
\end{rmrk}

The above remark motivates us to make the following conjecture:
\begin{conj}\label{conj:gllt}
Let $g_\varepsilon\rightharpoonup g$ weakly in $\hmo$ and let
$v_\varepsilon\in\hoz$, the weak solution of \eqref{eq:gvardata},
be such that $v_\ve \rightharpoonup v_0$ weakly in
$\hoz$, where $v_0\in\hoz$ is the unique solution of
\eqref{eq:limsoe}. Then 
\begin{equation}\label{eq:gllt}
\liminf_{\ve \rightarrow 0}\int_\Omega B_\ve \nabla
v_\ve\cdot\nabla v_\ve\, dx \ge \int_\Omega B^\sharp\nabla v_0
\cdot\nabla v_0\,dx.
\end{equation}
\qed
\end{conj}
The conjecture is open,
in general. However, in this section we prove the conjecture for
the periodic case with additional hypothesis on the data $g_\ve$.

To do so, we recall the periodic set up.
Let $Y=(0,1)^n$ be the reference cell in $\rn$. Let $A=A(x,y)\in
M(\alpha,\beta,\Omega\times Y)$, $Y$-periodic in $y$.
We assume that $A(x,y)=(a_{ij}(x,y))$ is in the class of
admissible functions. The corrector functions $\chi_i$, for $1\le i \le n$, is defined as
the solution of the cell problem
\begin{equation}\label{eq:celleqn}
\left\{ \begin{array}{rll}
-\textup{div}_y\left(A(x,y)[\nabla_y \chi_i(x,y) + e_i]\right) &  =
0 & \textup{in $Y$}\\
y \mapsto \chi_i(x,y)  &  & \textup{is $Y$-periodic},
\end{array}\right.
\end{equation}
where $\{e_1,\cdots,e_n\}$ is the standard basis of
$\mathbb{R}^n$. Let us define the corrector matrix as follows;
$P(x,y)e_i=\nabla_y\chi_i(x,y)$. It has been shown that $A_\ve$
$H$-converges to $A_0$ which is given by,
\begin{equation}\label{eq:ijentry}
(A_0)_{ij} = \int_Y A(x,y)(P(x,y)e_i + e_i)\cdot (P(x,y)e_j+e_j)\,dy.
\end{equation}
Let $B=B(x,y)\in M(c,d,\Omega\times Y)$
belong to the admissible class of functions. Assume $B$ is
symmetric. Let us now define $B^\sharp$ to be the matrix whose $ij^\text{th}$ entry is given by,
\begin{equation}\label{eq:bijentry}
(B^\sharp)_{ij} = \int_Y B(x,y)(P(x,y)e_i+e_i)\cdot (P(x,y)
e_j+e_j)\,dy
\end{equation}
We shall now recall a $H$-convergence result for weak
data proved in \cite{tmkakn} using two-scale convergence.
\begin{thm}[{\cite[Theorem 2.1]{tmkakn}}]\label{thm:wkdt}
Let $\gamma<1$ be a fixed real number. Let $v_\varepsilon\in\hoz$
be the weak solution of 
\[
\left\{ \begin{array}{rll}
-\textup{div}(\aepsper\nabla v_\ve) &  =  g_\ve & \textup{in
$\Omega$}\\
v_\ve &  =  0  & \textup{on $\bdry$},
\end{array}\right.
\]
where $g_\ve\in\ltoo$ is such that $g_\varepsilon\rightharpoonup
g$ weakly in $\hmo$ and $\ve^\gamma g_\varepsilon$ is bounded in
$\ltoo$. Then,  
\[
\left.
\begin{array}{r}
v_\ve \rightharpoonup v_0 \textup{ weakly in } \hoz\\
\aeps\nabla v_\ve \rightharpoonup A_0\nabla v_0 \textup{ weakly
in } (L^2(\Omega))^n.
\end{array}
\right\}
\]
is satisfied, where $v_0\in\hoz$ is the unique
solution of \eqref{eq:limsoe} and $A_0$ is as given in \eqref{eq:ijentry}.\qed
\end{thm}
We now prove a version of Conjecture \ref{conj:gllt} for the
periodic case.
\begin{prop}\label{prop:gllt}
If $g_\ve, g$ and $v_\ve$ satisfy the hypothesis as in Theorem
\ref{thm:wkdt}, then the inequality, as given in \eqref{eq:gllt},
holds for $B^\sharp$ as given in \eqref{eq:bijentry}.
\end{prop}
\begin{proof}
Since, $g_\ve$ weakly converges to $g$ in $\hmo$, we have
$v_\ve$ bounded in $\hoz$. Thus, by the compactness of two-scale
convergence, there exists $v_0\in\hoz$ and $v_1\in L^2[\Omega;
H^1_{\text{per}}(Y)]$ such that 
\[
\nabla v_\ve \stackrel{2\text{s}}{\rightharpoonup} \nabla
v_0+\nabla_y v_1(x,y).
\]
Moreover, by Proposition \ref{prop:convexlsc}, we have
\begin{equation}\label{eq:lsconvx}
\liminf_{\ve\to 0}\int_\Omega \beps\nabla v_\ve\cdot\nabla v_\ve\,
dx  \ge   \int_{\Omega\times Y}B(x,y)\left[\nabla v_0 +
\nabla_y v_1\right]\cdot\left[\nabla v_0 + \nabla_y v_1\right]
\,dy\,dx.
\end{equation}
It was shown in the proof of Theorem \ref{thm:wkdt} that,
\[
v_1(x,y)=\sum_{i=1}^n \chi_i(x,y)\frac{\partial v_0}{\partial
x_i}(x)
\]
and therefore, $\nabla_y v_1(x,y) = P(x,y)\nabla v_0$.
Thus,
\[
\liminf_{\ve\to 0}\int_\Omega \beps\nabla v_\ve\cdot\nabla v_\ve\, dx
\ge  \int_{\Omega\times Y}B(x,y) (P(x,y)+I)\nabla v_0\cdot
(P(x,y)+I)\nabla v_0 \,dy\,dx
\]
and by using \eqref{eq:bijentry}, we have
\[
\liminf_{\ve\to 0}\int_\Omega \beps\nabla v_\ve\cdot\nabla v_\ve\, dx
\ge \int_\Omega B^\sharp \nabla v_0\cdot \nabla v_0\,dx
\]
Thus, we have shown \eqref{eq:gllt} in the periodic case with
$B^\sharp$ as defined in \eqref{eq:bijentry}.
\end{proof}
To keep the proof of above proposition self-contained, to an
extent, we derive in the following lemma the inequality \eqref{eq:lsconvx} using some convex analysis
arguments. We follow the line of argument as given in the
proof of \cite[Proposition 2.5]{bouchfragala}.
\begin{lemma}
Let $\nabla v_\ve \stackrel{2\textup{s}}{\rightharpoonup} \nabla
v_0+\nabla_y v_1(x,y)$ where $v_0\in\hoz$ and $v_1\in L^2[\Omega;
H^1_{\textup{per}}(Y)]$ and let $B(x,y)\in M(c,d,\Omega\times Y)$,
$Y$-periodic in $y$, be a symmetric matrix, then \eqref{eq:lsconvx} is valid.
\end{lemma}
\begin{proof}
Let $\Phi\in \left(\mathcal{D}[\Omega;
C^\infty_{\text{per}}(Y)]\right)^n$.
Then, by Fenchel's inequality (for quadratic forms),
\begin{eqnarray*}
I_\ve := \frac{1}{2}\int_\Omega \beps\nabla v_\ve\cdot\nabla v_\ve\, dx
& \ge & \int_\Omega \nabla v_\ve\cdot\Phieps \,dx\\
& &  - \frac{1}{2}\int_\Omega
B^{-1}\left(x,\frac{x}{\varepsilon}\right)\Phieps\cdot\Phieps\,dx
\end{eqnarray*}
where $B^{-1}$ denotes the inverse of $B$ (which exists). We take
$\liminf$ on both sides of the inequality. By two-scale
convergence of $v_\ve$ the first term on right hand side becomes,
\[
\lim_{\ve\to 0}\int_\Omega \nabla v_\ve\cdot\Phieps \,dx =
\int_{\Omega\times Y}\left[\nabla v_0 + \nabla_y
v_1(x,y)\right]\cdot\Phi(x,y) \,dy\,dx.
\]
Now, since $B$ and $\Phi$ are in the admissible class of
functions, so is $B^{-1}\Phi$. Therefore, using
Theorem~\ref{thm:strweak}, we have
\[
\lim_{\ve \to 0} \int_\Omega
B^{-1}\left(x,\frac{x}{\varepsilon}\right)\Phieps\cdot\Phieps\,dx
 =  \int_{\Omega\times Y}
B^{-1}(x,y)\Phi(x,y)\cdot\Phi(x,y)\,dy\,dx.
\]
Thus, we conclude
\begin{eqnarray*}
\liminf_{\ve\to 0} I_\ve & \ge &  \int_{\Omega\times Y}\left[\nabla v_0 + \nabla_y
v_1(x,y)\right]\cdot\Phi(x,y) \,dy\,dx\\
& & - \frac{1}{2}\int_{\Omega\times Y}
B^{-1}(x,y)\Phi(x,y)\cdot\Phi(x,y)\,dy\,dx.
\end{eqnarray*}
Taking supremeum over all $\Phi\in\left(\mathcal{D}(\Omega\times
Y)\right)^n$ on the right hand side and using Lemma~\ref{lm:commute}, we obtain,
\begin{eqnarray*}
\liminf_{\ve\to 0} I_\ve & \ge &  \int_{\Omega\times Y}
\sup_{\xi\in\rn}\left\{\left[\nabla v_0 + \nabla_y
v_1(x,y)\right]\cdot\xi - \frac{1}{2}
B^{-1}(x,y)\xi\cdot\xi\right\}\,dy\,dx.
\end{eqnarray*}
Thus, by \eqref{eq:conconj}, we get
\[
\liminf_{\ve\to 0}\int_\Omega \beps\nabla v_\ve\cdot\nabla v_\ve\, dx
 \ge   \int_{\Omega\times Y}B(x,y)\left[\nabla v_0 + \nabla_y
v_1\right]\cdot\left[\nabla v_0 + \nabla_y v_1\right]
\,dy\,dx.
\]
Thus, we have shown \eqref{eq:lsconvx}.
\end{proof}

\section{Dirichlet-type cost functional --- Periodic
Case}\label{sec:bprob}

The purpose of this section is to announce the complete solution of a
problem considered in \cite[\S 3]{tmkakn}. More
precisely, we improve \cite[Theorem
3.7]{tmkakn} in its full generality with no assumptions on the
  control set.  We shall restrict
ourselves to the non-perforated case, for simplicity. However, the
results remain valid in perforated case with necessary
modifications.

The matrices $A(x,y)$ and $B(x,y)$ are periodic in $Y$ and is as given in the
previous section. Also recall that $B$ is symmetric. The corrector
matrix $P$ is as defined in the line following \eqref{eq:celleqn}. Let $U$ be a closed convex subset of $\ltoo$ and $f\in\ltoo$. Given
$\theta\in U$, the cost functional $\jeps$ is defined as,
\begin{equation}\label{eq:bcstfn}
\jeps(\theta)=
\frac{1}{2}\int_\Omega\beps\nabla\ueps\cdot\nabla\ueps\,dx +
\frac{\ve}{2}\|\theta\|^2_2
\end{equation}
where $\ueps\in \hoz$ is the unique solution of 
\begin{equation}\label{eq:gensteqn}
\left\{ \begin{array}{rll}
-\textup{div}(\aepsper\nabla\ueps) &  =  f + \theta & \textup{in }
\Omega\\
\ueps &  =  0  & \textup{on }\bdry.
\end{array}\right.
\end{equation}
Let $\optcon$ be the optimal controls and let $\optstat$ be the
state corresponding to $\optcon$.
Using arguments similar to that of \cite[Lemma 3.2]{tmkakn}, we
observe that $\optcon$ is bounded in $\hmo$ and there exists a
$\theta^*\in\hmo$ such that, for a subsequence, 
\[
\optcon\rightharpoonup\theta^* \text{ weakly in }\hmo
\]
and
\begin{equation}\label{eq:optcgce}
\varepsilon^{1/2} \optcon\rightharpoonup 0 \text{
weakly in }\ltoo.
\end{equation}
Therefore, by Theorem~\ref{thm:wkdt}, $\optstat$ converges weakly
in $\hoz$ to $u^*\in\hoz$ solving,
\[
\left\{ \begin{array}{rll}
-\textup{div}(A_0\nabla u^*) &  =  f+\theta^* & \textup{in
$\Omega$}\\
u^* &  =  0  & \textup{on $\bdry$}.
\end{array}\right.
\]
Let $V$ be the weak closure of $U$ in $\hmo$. By the convexity of
$U$, $V$ is also the strong closure in $\hmo$. We set
\[
F_\ve(\theta)=\left\{
\begin{array}{cl}
\jeps(\theta) & \text{ if } \theta\in U\\
+\infty & \text{ if } \theta\in \hmo\setminus U;
\end{array}
\right.
\]
where $\jeps$ is as given in \eqref{eq:bcstfn}. Let the matrix
$B^\sharp$ be as defined in \eqref{eq:bijentry} and let
$u=u(\theta)\in\hoz$ be the weak solution of
\[
\left\{ \begin{array}{rll}
-\textup{div}(A_0\nabla u) &  =  f+\theta & \textup{in
$\Omega$}\\
u &  =  0  & \textup{on $\bdry$}.
\end{array}\right.
\]
We set
\[
F(\theta)=\left\{
\begin{array}{cl}
\displaystyle{\frac{1}{2} \int_\Omega B^\sharp \nabla u(\theta)\cdot \nabla
u(\theta)\,dx}& \text{ if } \theta\in V\\
+\infty & \text { if } \theta\in \hmo\setminus V.
\end{array}
\right.
\]
The functional $F$ is coercive in $\hmo$ and thus, there is
exists a unique minimizer in $V$. We now show that this minimizer
is none other than $\theta^*$.

\begin{thm}\label{thm:bgamma}
$F_\ve$ $\Gamma$-converges to $F$ in the weak topology of $\hmo$.
Furthermore, $\theta^*$ is the minimizer of $F$ and $F_\ve(\optcon)
\rightarrow F(\theta^*)$.
\end{thm}
\begin{proof}
Let $\theps$ be a sequence weakly converging to $\theta$
in $\hmo$. Observe that it is enough to consider the case when
$\theta\in V$ and $\theps\in U$, for infinitely many $\ve$, as the
other cases correspond to the trivial situation ($\liminf F_\ve(\theps)$ is
infinite). If $\theta\in V$ and $\theps \in U$, we have,
\begin{eqnarray*}
\liminf_{\ve\to 0}F_\ve(\theps) & = & \liminf_{\ve\to
0}\left[\frac{1}{2} \int_\Omega \beps\nabla \ueps\cdot\nabla
  \ueps\, dx + \frac{\ve}{2}\|\theps\|^2_2\right]\\
& \ge & \liminf_{\ve\to 0}\frac{1}{2} \int_\Omega \beps\nabla
\ueps\cdot\nabla \ueps\, dx.
\end{eqnarray*}
Observe that $\{\ve^{\frac{1}{2}}\theps\}$ is bounded in $\ltoo$ and
converges weakly to $0$. Using the inequality \eqref{eq:gllt} as proved in Proposition~\ref{prop:gllt}, we have 
\[
\liminf_{\ve\to 0}F_\ve(\theps) \ge \frac{1}{2} \int_\Omega
B^\sharp \nabla u(\theta)\cdot \nabla u(\theta)\,dx= F(\theta).
\]
It now remains to prove the $\limsup$-inequality. It is enough to
consider the case $\theta\in V$ (the finite situation). By the
density of $U$ in $V$ and convexity of $U$, there exists a
sequence $\theps$ \emph{strongly} converging to $\theta$ in
$\hmo$. Therefore, by Theorem~\ref{thm:genenercon}, we have,
\[
\limsup_{\ve\to 0}F_\ve(\theps) = F(\theta).
\]
Thus, we have shown the $\Gamma$-convergence of $F_\ve$ to $F$.
Therefore, by Theorem~\ref{thm:gamma}, $\theta^*$ is the minimizer of $F$ and
$F_\ve(\optcon) \rightarrow F(\theta^*)$.
\end{proof}

It follows from Proposition~\ref{prop:gllt} that for the weakly
converging optimal controls,  $\optcon \rightharpoonup \theta^*$
in $\hmo$, \eqref{eq:gllt} holds for the corresponding optimal
states $\optstat$. However, in the following proposition we
obtain the equality in \eqref{eq:gllt}.

\begin{lemma}\label{lm:equality}
For the optimal states $\optstat$ and $u^*$, we have
\[
\lim_{\ve\to 0}\int_\Omega \beps \nabla \optstat\cdot\nabla
\optstat\, dx = \int_\Omega B^\sharp\nabla u^* \cdot\nabla u^*\,dx
\]
\end{lemma}
\begin{proof}
A consequence of Theorem~\ref{thm:bgamma} is that $\theta^*$ is a minimizer of $F$ and
\[
\jeps(\optcon) \rightarrow \frac{1}{2}\int_\Omega B^\sharp\nabla
u^* \cdot\nabla u^*\,dx.
\]
Observe that,
\[
\frac{1}{2}\int_\Omega \beps \nabla \optstat\cdot\nabla \optstat\,
dx \le \jeps(\optcon).
\]
Now, taking $\limsup$ both sides we have,
\[
\limsup_{\ve\to 0}\int_\Omega \beps \nabla \optstat\cdot\nabla
\optstat\, dx \le \int_\Omega B^\sharp\nabla u^* \cdot\nabla u^*\,dx
\]
and comparing with \eqref{eq:gllt}, gives the equality in
\eqref{eq:gllt} for optimal states.
\end{proof}
\begin{rmrk}
We now observe that one can, in fact, improve the convergence in
\eqref{eq:optcgce} to strong convergence. Note that, as a consequence of
Theorem~\ref{thm:bgamma} and Lemma~\ref{lm:equality},
\[
\lim_{\ve\to 0}\frac{\ve}{2}\|\optcon\|^2_2 = \lim_{\ve\to
0}\left(F_\ve(\optcon)- \frac{1}{2} \int_\Omega \beps\nabla
  \optstat\cdot \nabla \optstat\, dx \right) = 0.
\]
\end{rmrk}\qed

In general, there is no corrector result available for weakly
converging data. However, we now prove a corrector result for the
optimal states.
\begin{thm}\label{thm:corrector}
Let the corrector matrix $P(x,y)$ be as defined in
\eqref{eq:celleqn}. We also assume that both $A$ and $B$ are in
$C[\Omega;L^\infty_{\textup{per}}(Y)]^{n\times n}$. Then,
\[
\nabla\optstat - \left[\peps+I\right]\nabla u^* \rightarrow 0 \text{ strongly in
} \ltoo.
\]
\end{thm}
\begin{proof}
Let
\[
I_\ve = \int_\Omega \beps\left[\nabla\optstat -
\left[\peps+I\right]\nabla
u^*\right]\cdot\left[\nabla\optstat - \left[\peps+I\right]\nabla
u^*\right]\, dx.
\]
Then (also using the symmetry of $B$),
\begin{eqnarray*}
I_\ve & = & \int_\Omega \beps \nabla \optstat\cdot\nabla
\optstat\, dx - 2 \int_\Omega \beps\left[\peps+I\right] \nabla u^*\cdot\nabla \optstat\, dx\\
& & + \int_\Omega \beps\left[\peps+I\right]\nabla
u^*\cdot\left[\peps+I\right]\nabla u^*\, dx.
\end{eqnarray*}
To simplify notation, we rewrite above equation as,
\[
I_\ve  = I^1_\ve -2 I^2_\ve + I^3_\ve.
\]
Using Lemma~\ref{lm:equality}, we see that
\[
\lim_{\ve\to 0}I^1_\ve = \int_\Omega B^\sharp\nabla u^* \cdot\nabla u^*\,dx.
\]
Since $A\in C[\Omega;L^\infty_{\text{per}}(Y)]^{n\times n}$ by
the continuous dependence on data for elliptic equation, we have
$P\in C[\Omega;L^2_{\text{per}}(Y)]^{n\times n}$ and hence
$BP$ in the class of admissible functions and thus, strongly
two-scale converges. By Theorem~\ref{thm:strweak}
\begin{eqnarray*}
\lim_{\ve\to 0}I^2_\ve & = & \int_{\Omega\times Y}B(x,y)(P(x,y)+I)\nabla
u^*\cdot (P(x,y)+I)\nabla u^* \,dy\,dx\\
& = & \int_\Omega B^\sharp\nabla u^*\cdot\nabla u^*\,dx 
\end{eqnarray*}
Similarly, we compute the last term,
\begin{eqnarray*}
\lim_{\ve\to 0}I^3_\ve & = & \int_{\Omega\times Y}B(x,y)(P(x,y)+I)\nabla
u^*\cdot (P(x,y)+I)\nabla u^* \,dy\,dx\\
& = & \int_\Omega B^\sharp\nabla u^*\cdot\nabla u^*\,dx
\end{eqnarray*}
Therefore, by the coercivity of $B$,
\[
c \left\|\nabla\optstat - \left[\peps+I\right]\nabla
u^*\right\|^2_2 \le I_\ve.
\]
Now taking $\limsup$ both sides, we have our desired result.
\end{proof}

Recall from Remark~\ref{rm:econ} that the energy convergence
\eqref{eq:energy} is not always true for weakly converging data.
Besides Lemma~\ref{lm:equality}, we have the following corollary.
\begin{cor}
For the optimal states $\optstat$ and $u^*$, we have
\[
\lim_{\ve\to 0}\int_\Omega \aeps \nabla \optstat\cdot\nabla
\optstat\, dx = \int_\Omega A_0\nabla u^* \cdot\nabla u^*\,dx
\]
\end{cor}
\begin{proof}
The sequence of corrector matrices $P_\ve$ introduced in
\S\ref{sec:hcon} can be taken to be $P_\ve = \peps +I$. Therefore,
by Theorem~\ref{thm:corrector}, we have
\begin{eqnarray*}
\int_\Omega \aeps \nabla \optstat\cdot\nabla
\optstat\, dx & = & \int_\Omega \aeps P_\ve \nabla u^*\cdot P_\ve
\nabla u^*\, dx + o(1)\\
& = & \int_\Omega P_\ve^t \aeps P_\ve \nabla u^*\cdot \nabla u^*\, dx+
o(1)\\
& = & \int_\Omega A_0\nabla u^* \cdot\nabla u^*\,dx + o(1)
\end{eqnarray*}
where the last equality is due to the properties of corrector
matrices.
\end{proof}

\begin{rmrk}
More generally, given $A_\ve$ $H$-converges to $A_0$,
$g_\varepsilon\in\hmo$, let us assume that $v_\varepsilon$, the solution
 of \eqref{eq:gvardata}, satisfies the convergences in
\eqref{eq:staweakcon} and \eqref{eq:fluxcon}. Then, the energy
convergence  
\[
\int_\Omega A_\ve\nabla\veps.\nabla\veps\,dx
\stackrel{\ve\to 0}{\longrightarrow}
\int_\Omega A_0\nabla v_0.\nabla v_0\,dx
\]
is, in fact, equivalent to the convergence
\[
\nabla v_\ve - P_\ve \nabla v_0 \rightarrow 0 \text{ strongly in } \ltoo
\]
under certain hypotheses, for example, $A_\ve$ symmetric.
\end{rmrk}

\section{Measure data}\label{sec:measdt}

A \emph{Borel measure} on $\Omega$ is a countably additive set
function defined on the Borel subsets of $\Omega$ with values in
$[-\infty,+\infty]$. The total variation of a measure
$\lambda$ is denoted by $|\lambda|$. Let $\borelmeas$ denote the
set of all Borel measures $\lambda$ with finite variation, i.e., $|\lambda|(\Omega) < +\infty$. We say a subset $\mathcal{E}$ in
$\borelmeas$ is bounded if we have $\sup_{\lambda\in
\mathcal{E}}|\lambda|(\Omega) < +\infty$.

Given $A\in\mabomega$, we make precise the notion of solution when
$\lambda\in\borelmeas$ for the second order linear elliptic equation
\begin{equation}\label{eq:meastate}
\left\{ \begin{array}{rll}
-\textup{div}(A(x)\nabla u) &  =  \lambda & \textup{ in }\Omega\\
u &  =  0  & \textup{ on }\partial\Omega.\\
\end{array}\right.
\end{equation}
Assume $p>n$, by the classical Sobolev Imbedding
$W^{1,p}_0(\Omega)\subset C_0(\Omega)$, we have
$\borelmeas\subset W^{-1,q}(\Omega)$, where $q=\frac{p}{p-1}$ is the conjugate exponent of $p$. In this case, $u$ is defined to be the
usual variational solution for elliptic equations. For instance,
when $p=2$ and $n<2$, $\lambda\in \hmo$ and we have the solution $u\in\hoz$ given as
\[
\int_\Omega A(x)\nabla u\cdot \nabla w \,dx = \left\langle
\lambda, w\right\rangle,  \quad \forall w\in\hoz.
\]
However, for $n>2$, $\lambda$ is not necessarily in $\hmo$ and
the solution to \eqref{eq:meastate} cannot be considered in the above
sense. For
this situation G.~Stampacchia introduced a notion of solution in
\cite[Definition 9.1]{stampacchia} using duality. We shall set $n\ge3$ (to avoid the trivial case)
for this section and the next section. Let $1^\star=\frac{n}{n-1}$
denote the Sobolev conjugate of the exponent $1$.
\begin{defn}\label{def:stmpsoln}
Given $\lambda\in\borelmeas$, we say $u\in L^1(\Omega)$ is a \emph{Stampacchia solution} of
\begin{equation}\label{eq:stmpsoln}
\left\{ \begin{array}{rll}
-\textup{div}(A(x)\nabla u) &  =  \lambda & \textup{ in }\Omega\\
u &  =  0  & \textup{ on }\partial\Omega\\
\end{array}\right.
\end{equation}
whenever
\begin{equation}\label{eq:nsc}
\int_\Omega ug \,dx = \int_\Omega v\,d\lambda, \quad \forall
g\in\linfty,
\end{equation}
and $v$ solves
\begin{equation}\label{eq:adjeqn}
\left\{ \begin{array}{rll}
-\textup{div}({A^t}(x)\nabla v) &  =  g & \textup{ in }\Omega\\
u &  =  0  & \textup{ on }\partial\Omega.\\
\end{array}\right.
\end{equation}\qed
\end{defn}
The existence of $v$ is well known, by Lax-Milgram theorem, and
 \[
\|v\|_{\hoz} \le C_0 \|g\|_2
\]
where the constant $C_0$ depends on $n,\alpha$ and $\Omega$.
Henceforth $C_0$ will denote a generic constant depending on
$n,\alpha$ and $\Omega$. Assuming sufficient regularity of
$\partial\Omega$, by a classical regularity result (cf.
\cite{giltru}), given $s>n$ there exists $\nu$ and a constant
$C_0$ such that $v$ satisfies the H\"{o}lder estimate,
\begin{equation}\label{eq:holder}
\|v\|_{C^{0,\nu}(\Omega)} \le C_0 \|g\|_s.
\end{equation}
The existence, uniqueness and regularity of Stampacchia solution $u$ was shown in
\cite[Theorem 9.1]{stampacchia}. One has the following regularity
for $u$,
\begin{equation}\label{eq:qesti}
\|u\|_{\wlq} \le C_0|\lambda|, \quad \forall \allq
\end{equation}
Observe that the variational solution corresponding to a data
$h\in\ltoo$ is a Stampacchia solution corresponding to the
measure $\lambda = h\,dx$, induced by $h$.

We now show the asymptotic behaviour of the Stampacchia
solution under weak-* convergence in $\borelmeas$ of the data. The asymptotic
behaviour of Stampacchia solution was observed in
\cite{malusaorsina} for a fixed $\lambda\in\borelmeas$. The
argument, however, quite simply extends to measures converging weak-* in
$\borelmeas$, which we give below for completeness sake. Let $\ve>0$ be a
given parameter which tends to zero and let $A_\ve\in\mabomega$ be
a family of matrices.

\begin{thm}[Asymptotic Behaviour]\label{thm:asybeh}
Let $\aeps$ $G$-converge to $A_0$
and let $\lambdaeps$ weak-* converge to $\lambda$ in $\borelmeas$. If
$\ueps\in L^1(\Omega)$ is the Stampacchia solution of
\begin{equation}\label{eq:vareqn}
\left\{ \begin{array}{rll}
-\textup{div}(\aeps\nabla \ueps) &  =  \lambdaeps & \textup{ in }\Omega\\
\ueps &  =  0  & \textup{ on }\partial\Omega,\\
\end{array}\right.
\end{equation}
then $\ueps \rightharpoonup u_0$ weakly in $\wlq$,  for all
$\allq$ where $u_0$ is
the Stampacchia solution of
\begin{equation}\label{eq:homeqn}
\left\{ \begin{array}{rll}
-\textup{div}(A_0(x)\nabla u_0) &  =  \lambda & \textup{ in }\Omega\\
u_0 &  =  0  & \textup{ on }\partial\Omega.\\
\end{array}\right.
\end{equation}
\end{thm}
\begin{proof}
We have $\lambdaeps$ bounded in $\borelmeas$, since $\lambdaeps$ converges
weak-* in $\borelmeas$. Thus, by \eqref{eq:qesti}, there is a constant $C_0$ (depending
on $n, \alpha$ and $\Omega$, and independent of $\{A_\ve\}$) such that,
\[
\|\ueps\|_{\wlq} \le C_0, \quad \forall \allq.
\]
 Consequently, there exists a subsequence (still denoted by
$\ve$), such that $\ueps \rightharpoonup u_0$ weakly in $\wlq$ for
$\allq$. For each $\ve$, we have by the definition of Stampacchia
solution, 
\begin{equation}\label{eq:varnc}
\int_\Omega \ueps g \,dx = \int_\Omega \veps \,d\lambdaeps, \quad \forall
g\in\linfty,
\end{equation}
where $\veps$ solves
\begin{equation}\label{eq:varadjeqn}
\left\{ \begin{array}{rll}
-\textup{div}({A^t_{\ve}}(x)\nabla \veps) &  =  g & \textup{ in }\Omega\\
\veps &  =  0  & \textup{ on }\partial\Omega.\\
\end{array}\right.
\end{equation}
By the equi-continuity of the sequence $\veps$ guaranteed by the
uniform H\"{o}lder estimate \eqref{eq:holder}, $\veps$ converges uniformly in $\Omega$ to
some $v_0$ (cf. \cite[section 1]{malusaorsina}) and by the theory of $G$-convergence $v_0$ is the unique solution of 
\begin{equation}\label{eq:homadjeqn}
\left\{ \begin{array}{rll}
-\textup{div}({A^t_0}(x)\nabla v_0) &  =  g & \textup{ in }\Omega\\
v_0 &  =  0  & \textup{ on }\partial\Omega.\\
\end{array}\right.
\end{equation}
One can pass to the limit in the right side of \eqref{eq:varnc},
\[
\lim_{\ve\rightarrow 0}\int_\Omega\veps\,d\lambdaeps =
\int_\Omega v_0\,d\lambda.
\]
Thus,
\[
\int_\Omega u_0 g \,dx = \int_\Omega v_0 \,d\lambda, \quad \forall
g\in\linfty.
\]
Therefore, by Definition~\ref{def:stmpsoln}, $u_0$ is the
Stampacchia solution of \eqref{eq:homeqn}.
\end{proof}

\section{Optimal Measures}\label{sec:optmsr}

Let us fix a $\allr$.
The motivation for the choice of range for $r$ is due to
the fact that there exists a $\allq$ such that $\wlq$ is compactly imbedded
in $L^r(\Omega)$, for $\allr$. For any given constant $k>0$, let $U$ be the set of all
$\ltoo$ functions such that their $L^1$-norms are bounded, i.e.,
\[
U =\left\{\theta\in\ltoo\mid \|\theta\|_1 \le k \right\}.
\]
Observe that $U$ is closed and convex in $\ltoo$. Given
$\theta\in U$ and $\{A_\ve\}\subset\mabomega$, we are interested in the asymptotic behaviour of the cost functional $\jeps$ defined as,
\begin{equation}\label{eq:ucstfn}
\jeps(\theta)=
\|\ueps(\theta)\|^r_r+ \ve\|\theta\|^2_2, 
\end{equation}
where $\ueps\in \hoz$ is the unique solution of
\begin{equation}\label{eq:state}
\left\{ \begin{array}{rll}
-\textup{div}(A_\ve\nabla \ueps) &  =  f + \theta & \textup{in $\Omega$}\\
\ueps &  =  0  & \textup{on $\partial\Omega$.}
\end{array}\right.
\end{equation}
For a fixed $\ve$, $\ueps \in L^r(\Omega)$ and
$\|\ueps(\theta)\|_r^r$, for all $\allr$, is continuous  as
function of $\theta$ for the weak topology in $\ltoo$. Thus, $\jeps$ is weakly
lower semicontinuous and is strictly convex. Therefore, there exists a
unique $\optcon\in U$ which minimizes $\jeps$ in $U$. Given any
$\theta\in\ltoo$, we associate with it a measure $\mu$, defined as follows:
\begin{equation}\label{eq:conversion}
\mu(\omega)=\int_\omega\theta\,dx, \quad \forall \text{ Borel set
}\omega \text{ in }\Omega
\end{equation}
and the state $\ueps\in\hoz$ is the Stampacchia solution of
\begin{equation}\label{eq:statemeas}
\left\{ \begin{array}{rll}
-\textup{div}(A_\ve\nabla \ueps) &  =  f + \mu & \textup{in $\Omega$}\\
\ueps &  =  0  & \textup{on $\partial\Omega$.}
\end{array}\right.
\end{equation}
Let $\optmu$ denote the measure corresponding to $\optcon$.
Since $U$ is bounded in $L^1(\Omega)$, by Banach-Alaoglu theorem,
there exists a measure $\mu^*\in\borelmeas$ such that, for a
subsequence, $\optmu$ weak-* converges to $\mu^*$ in $\borelmeas$. Let $V$ be the weak-* closure of $U$ in
$\borelmeas$. For any $\lambda\in V$, we define the functional $J:
V \to \mathbb{R}\cup\{-\infty,+\infty\}$, as
 \begin{equation}\label{eq:ulim}
J(\lambda) = \|u_0\|^r_r,
\end{equation}
where the state $u_0$ is the Stampacchia solution of
\begin{equation}\label{eq:statelim}
\left\{ \begin{array}{rll}
-\textup{div}(A_0\nabla u_0) &  =  f + \lambda & \textup{in $\Omega$}\\
u_0 &  =  0  & \textup{on $\partial\Omega$.}
\end{array}\right.
\end{equation}
We now extend our functionals $\jeps$ and $J$ to the space of
Borel measures with finite variations, $\borelmeas$, in the following
way:
\[
F_\ve(\mu)=\left\{
\begin{array}{cl}
\jeps(\theta) & \text{ if } \mu=\theta\,dx, \theta\in U\\
+\infty & \text{ otherwise }
\end{array}
\right.
\]
and
\[
F(\lambda)=\left\{
\begin{array}{cl}
J(\lambda) & \text{ if } \lambda\in V\\
+\infty & \text { if } \lambda\in \borelmeas\setminus V.
\end{array}
\right.
\]
\begin{thm}\label{thm:measgamma}
$F_\ve$ $\Gamma$-converges to $F$ in the weak-* topology of
$\borelmeas$. Furthermore, $\mu^*$ is the minimizer of $F$ and
$F_\ve(\optcon) \rightarrow F(\mu^*)$.
\end{thm}
\begin{proof}
Let $\mueps$ be a sequence weakly-* converging to $\mu$ in
$\borelmeas$. It is enough to consider the case when $\mu\in V$
and $\mueps\in U$, for infinitely many $\ve$, as the
other cases correspond to the trivial situation ($\liminf
F_\ve(\mueps)$ is infinite). Let $\theps$ be the
function associated with $\mueps$ as given in
\eqref{eq:conversion}. We have,
\begin{eqnarray*}
\liminf_{\ve\to 0}F_\ve(\mueps) & = & \liminf_{\ve\to
0}\left[\|\ueps\|^r_r + \ve\|\theps\|_2^2\right]\\
& \ge & \liminf_{\ve\to 0}\|\ueps\|^r_r. 
\end{eqnarray*}
For every $\allq$, $\{\ueps\}$ is bounded in $\wlq$. Therefore,
one can always choose a $s\in [1,1^\star)$ such that $r <
s^\star$. Thus, by Sobolev Imbedding, $W_0^{1,s}(\Omega)$ is compactly imbedded in $\lr$. Hence, there exists a $u_0$ such that, for a subsequence,
$\ueps$ strongly converges to $u_0$ in $\lr$.
Therefore, we have,
\[
\liminf_{\ve\to 0}F_\ve(\mueps)  \ge  \|u_0\|^r_r.
\] 
Also, by Theorem~\ref{thm:asybeh}, we have that $u_0$ is a
Stampacchia solution of \eqref{eq:statelim} corresponding to
$\mu$. Hence,
\[
\liminf_{\ve\to 0}F_\ve(\mueps)  \ge  F(\mu).
\]

It now only remains to prove the $\limsup$-inequality. To this
aim, we first prove the $\limsup$-inequality in $U$ and extend the
result to all of $V$ using the density of $U$ in $V$.

Let $F^+:= \Gamma\text{-}\limsup_{\ve\to 0}F_\ve$ in the weak-*
topology of $\borelmeas$. Given $\theta\in U$, we choose the constant
sequence $\theps=\theta$ in $U$. Thus,
\begin{eqnarray*}
F^+(\theta) \le \limsup_{\ve\to 0}F_\ve(\theps)  =
\limsup_{\ve\to 0}F_\ve(\theta)
& = & \limsup_{\ve\to 0}\left[\|\ueps\|^r_r + \ve\|\theta\|^2_2\right]\\
& = & \|u_0\|^r_r. 
\end{eqnarray*}
Therefore,
\begin{equation}\label{eq:inu}
F^+(\theta) \le F(\theta), \quad \forall \theta\in U.
\end{equation}
For any $\lambda\in V\setminus U$, there exists a sequence of
$\{\lambda_m\}\subset U$ such that $\lambda_m$ weak-* converges to
$\lambda$ in $\borelmeas$, as $m\to \infty$.
to�
Consider,
\begin{eqnarray*}
F^+(\lambda) & \le & \liminf_{m\to\infty}
F^+(\lambda_m) \qquad (\text{by the l.s.c of } F^+)\\
& \le & \lim_{m\to\infty} F(\lambda_m) \qquad (\text{by } \eqref{eq:inu})\\
& \le & F(\lambda) \qquad (\text{by the continuity of } F)
\end{eqnarray*}
Thus, we have \eqref{eq:inu} for all $\lambda\in V$ and hence,
the $\limsup$-inequality is proved. Therefore, $F_\ve$
$\Gamma$-converges to $F$. Moreover, by Theorem~\ref{thm:gamma},
$\mu^*$ is the minimizer of $F$ and $F_\ve(\optcon) \rightarrow F(\mu^*)$.
\end{proof} 
\begin{rmrk}
A consequence of Theorem~\ref{thm:measgamma} is that,
\[
\lim_{\ve\to 0}\ve\|\optcon\|^2_2 = \lim_{\ve\to
0}\left(F_\ve(\optcon)- \|\optstat\|^r_r\right) = 0.
\]
Thus, $\ve^{\frac{1}{2}}\optcon \rightarrow 0$ strongly in $\ltoo$.
\end{rmrk}\qed

\section*{Conclusion}
A question of interest is the regularity of the optimal controls
$\theta^*$ and $\mu^*$. To be precise, it would be of interest to
obtain conditions on $U$ or on the coefficients under which
$\theta^*$ or $\mu^*$ are in $\ltoo$.

The asymptotic behaviour of the optimal control problems given by
\eqref{eq:bcstfn}--\eqref{eq:gensteqn} has been proved for an
arbitrary control set in $\ltoo$, for the periodic case. A
question of further interest is whether these results can be
generalized to the non-periodic case too.
 
\section*{Acknowledgement} The authors wish to thank Guy
Bouchitt\'{e} for some fruitful discussions. The first author was
supported by FONDECYT Project No.3090052.

\bibliographystyle{h-elsevier2.bst}
\bibliography{reference}
\end{document}